\documentclass[12pt]{article}
\usepackage{amssymb, url, amsthm, amsmath}
\usepackage{amscd}
\newtheorem{theorem}{\noindent Theorem}
\newtheorem{lemma}{\noindent Lemma}
\newtheorem{definition}{\noindent Definition}
\newtheorem{corollary}{\noindent Corollary}
\newtheorem{conjecture}{\noindent Conjecture}
\newtheorem{statement}{\noindent Proposition}
\newtheorem{remark}{\noindent Remark}

\newcommand{\eps}{\varepsilon}
\newcommand{\dm}[2]{{\rm diam}_{#1}(#2)}
\newcommand{\esdm}[2]{{\rm essdiam}_{#1}(#2)}
\newcommand{\Ent}[2]{\mathbb{H}_{#1}(#2)}
\newcommand{\EntK}[2]{\mathbb{H}_{#1}^K(#2)}
\newcommand{\mnorm}[1]{\|#1\|_{m}}
\newcommand{\Lnorm}[1]{\|#1\|_{_{L^1(X^2)}}}
\def\tr{{\rm tr}\,}

\usepackage{showkeys}

\date{}
\author{A.~M.Vershik, P.~B.~Zatitskiy, F.~V.~Petrov}
\title{Geometry and dynamics of admissible metrics in measure spaces}

\begin{document}
\setcounter{section}{-1}
\maketitle

\begin{abstract}
We study a wide class of metrics in a Lebesgue space, namely the
class of so-called admissible metrics. We consider the cone of
admissible metrics, introduce a special norm in it, prove
compactness criteria, define the $\eps$-entropy of a measure space
with an admissible metric, etc. These notions and related results
are applied to the theory of transformations with invariant measure;
namely, we study the asymptotic properties of orbits in the cone of
admissible metrics with respect to a given transformation or a group
of transformations. The main result of this paper is a new
discreteness criterion for the spectrum of an ergodic
transformation: we prove that the spectrum is discrete if and only
if the $\eps$-entropy of the averages of some (and hence any)
admissible metric  over  its trajectory is uniformly bounded.

\end{abstract}

\tableofcontents

\section{Introduction}

This paper contains a number of results obtained in the framework of
the program outlined by the first author in
\cite{Ur, Mark, St.P} and concerning the asymptotic dynamics of
metrics in measure spaces and its applications to ergodic theory.
In the first chapter, we study the space of so-called admissible
metrics on a standard measure space; then, in the second chapter, we
use the developed machinery to characterize systems with discrete
spectrum in terms of scaling entropy. The main idea of our approach is as
follows. Consider an action of a countable group $G$ of measurable
transformations in a standard (Lebesgue) space $(X,\mu)$
with a continuous measure and assume that we are given a measurable
(regarded as a function of two variables) metric or semimetric
$\rho$ such that the corresponding metric space
structure on $X$ agrees with the measure
space structure (such a metric is called admissible,
see below). We iterate the
metric using the transformation group $G$ and consider
the averages of these
iterations over finite subsets of $G$ chosen in a special way
(for instance, over F{\o}lner sets in amenable groups):
$$\rho_n^{av}(x,y)=\frac{1}{\#A_n}\sum_{g\in A_n}\rho(gx,gy).$$
For the group ${\mathbb Z}$ with an
automorphism $T$ as a generator, we
have
$$\rho_n^{av}(x,y)=\frac{1}{n}\sum_{i=0}^{n-1}\rho(T^i x,T^i y).
$$

We suggest to study the asymptotic behavior (as
$n\rightarrow \infty$) of this sequence of metrics and its invariants,
and to find those invariants that do not depend on the choice of the
initial metric. The first example of such an invariant is the
$\eps$-entropy of the corresponding metric measure space, more
exactly, the scaling entropy. A general principle, which we justify in this
paper in the simplest case of a discrete spectrum action of an Abelean
group (in particular, $\mathbb{Z}$), is that
these asymptotic characteristics do not depend on the choice of the
initial (semi)metric $\rho$, at least for a wide range of
metrics, and thus are
ergodic invariants of the action. Most probably, this is also true in
many other cases.

Of course, the limit mentioned above does exist almost everywhere by the ergodic theorem
(applied to the square of the action on the space
$X\times X$) and is an invariant metric. But in most interesting
cases, namely, when the orthogonal complement to the constants
has no discrete spectrum, this limit metric is constant almost
everywhere, so that it determines the discrete topology on $X$ and hence is
not admissible in our sense. However, we will be interested not in the
limit itself, but in the asymptotic behavior of the average metrics.

The relation between scaling and classical entropies is easy to explain.
The Kolmogorov entropy of an automorphism $T$ in Sinai's definition is the
limit of the entropy of the product of $n$
rotations of a generating
partition for $T$ normalized by $n$. If this entropy vanishes, no change of
the normalization would give a new invariant. In our approach, we
suggest to consider the normalized limits of the {\it $\eps$-entropy}; and
the difference with the classical approach
is that we consider not the product of partitions,
but the average metric. This allows us to define the
asymptotics of the $\eps$-entropy also in the
case where the Kolmogorov entropy vanishes. The corresponding growth
(in $n$, for small $\eps$) is determined by the so-called scaling
sequence, and if the numerical limit does exist, then it is called the
{\it scaling entropy}, see \cite{usp}. As a very special case, this
notion includes also topological entropy. Some nontrivial examples
for actions of groups of the form $\sum_1^{\infty} {\Bbb Z}_2$ were
studied earlier (see \cite{usp,VGor})\footnote{In those papers, the
problem arose in connection with the theory of filtrations and the
pasts of Markov processes.}.

Similar suggestions, in different contexts and different generality, were
studied earlier. Feldman
(\cite{Fel}; see also the later papers \cite{Kat} and especially \cite{Fer},
where this problem is considered
from the point of view of complexity theory)
perceived the role of $\eps$-entropy (without using this term). An important difference
of our suggestion from all these papers is that instead of the theory
of measurable partitions (i.e., discrete semimetrics) we use the
theory of general admissible semimetrics and consider the
operation of averaging metrics, which has no simple interpretation
in terms of partitions (see the formula above). Averaging is much
more natural for ergodic theory than taking the maximum of
metrics. In another context, this operation was used in
\cite{OW}. Considering scaling sequences for the $\eps$-entropy
of automorphisms will make it possible
to classify the ``measure of chaoticity''
--- from the absence of growth (in the
case of discrete spectrum) up to linear growth (in the case of
positive Kolmogorov entropy). In between there must be classes
of automorphisms with zero Kolmogorov entropy but different scaling
entropy. In more traditional (probabilistic) terms, one might say that
we suggest to consider the asymptotics of sequences of Hamming-like
metrics in the space of realizations of a stationary random process.

Questions about more involved geometric invariants of sequences of
metrics apparently were not even posed. What is the difference
between the sequence of average metrics on a measure space constructed
from a Bernoulli automorphism and that constructed from a
non-Bernoulli $K$-automorphism? The growth of the $\eps$-entropy (the
scaling sequence) in these cases is the same; therefore, to
distinguish between them, one need to consider invariants not of a
single metric, but of several consecutive metrics.

To formulate very briefly the idea of the approach to ergodic theory suggested in
\cite{Ur, Mark, St.P}, it is to study {\it random stationary sequences
of admissible metrics on a given measure space and their asymptotic
invariants}, in contrast to the traditional probabilistic
interpretation of this theory as the study of stationary sequences of
random variables. It is quite obvious that the information
on the shifts contained in metrics is easier to extract than
that contained in functions of one variable, and
this allows one to hope for a simplification of the whole theory.

The results presented in the first chapter of the paper
is devoted to preliminary considerations, namely, to the study
of admissible metrics on a measure space. On the one hand,
Gromov's remarkable work (see \cite{Riem}) initiated a systematic study of
so-called $mm$-spaces (which in \cite{Ur} were called Gromov triples, or metric triples).
The most important fact here is the
reconstruction, or classification, theorem of Gromov and Vershik, about
a complete system of invariants of nondegenerate $mm$-spaces (see
\cite{Riem,umn04} and below), which is a particular case of the
classification theorem for measurable functions of several variables
\cite{Clas}. On the other hand, starting from the first author's papers
\cite{Ur,ran}, the following point of view on $mm$-spaces is suggested: in
contrast to the classical approach, where one fixes a topological
space (for instance, a metric compact space) and considers various
Borel measures on it, here, on the contrary, one fixes a
$\sigma$-algebra and a measure and varies {\it admissible metrics} on
this measure space. It is interesting that within this approach,
even the notion of a (semi)metric
needs to be slightly modified (fortunately, in a harmless way:
``an almost metric is a metric''). We consider in detail several
equivalent definitions of an admissible metric, which
are heavily used in what follows and underlie the whole approach. The
admissibility of a metric on a measure space means merely that it is
measurable and separable. The original measure is Borel with respect to
any admissible metric, and the completion of the original space with respect to
an admissible metric is a Polish space with a nondegenerate Borel
measure. There are many reformulations of the notion of admissibility,
including those involving matrix distributions, projective limits, etc. We
consider summable metrics; the space (cone) of admissible metrics lies
in $L^1(X\times X,\mu \times \mu)$ and is equipped with a special norm
(called the m-norm). The convergence in this norm is a
``convergence with a regulator,''
which appears in the theory of partially ordered Banach spaces.
We prove a number of properties of this norm and an important
 compactness criterion for a family of metrics in this norm,
which is a generalization of the Kolmogorov--Riesz compactness
criterion for $L^1$. In one of the sections we discuss
how an admissible metric can be restricted to the elements of a
measurable partition. This question is related to a serious problem
about the correctness of the restriction of a measurable function of
two or more variables to a subset of smaller dimension.

The main result of this paper (the second chapter) illustrates this idea; namely, it
says that {\it for an action of $\Bbb Z$ (and discrete Abelean
groups), the spectrum is discrete if and only if for some (and
hence any) admissible metric, the $\eps$-entropy of its averages is
bounded}. This criterion does not require explicit calculation of the
spectrum or even (as in Kushnirenko's criterion; see the last section)
enumeration of the asymptotics of all possible sequences of entropies,
etc. It suffices to perform calculations only for one admissible metric.
A similar result in a more special situation was obtained by another
method in \cite{Fer, FK}.

In the last section, we discuss relations of our results with the characterization of
discrete spectrum systems in terms of Kirillov--Kushnirenko A-entropy
 (or sequential entropy) \cite{Kush} and Kushnirenko's
compactness criterion for a set of partitions.
The difference between our
approaches is that we consider the $\eps$-entropy of the averages of
consecutive iterations of a metric rather than the normalized entropy
of the supremum over subsequences of partitions, as in
\cite{Kush}. We formulate several open problems and conjectures.

\medskip
The paper is supported by the Russian Federation Government grant
11.G34.31.0026 and the RFBR grants
11-01-12092-ofi-m and
11-01-00677-a, as well as by the Chebyshev Laboratory of the
St.~Petersburg State University.
The authors express their deep gratitude to N.Tsilevich for the translation
of the paper and for useful remarks.

\section{The geometry of admissible metrics}

\subsection{Definitions of admissible metrics on measure spaces}

Let $(X,\mu)$ be a Lebesgue space. We will be
mainly interested in spaces with a normalized (i.e., such that $\mu(X)=1$) continuous positive
measure, but all definitions apply to an arbitrary Lebesgue space,
in which the measure may contain atoms.

\begin{definition}
A metric or semimetric $\rho$ on the space $X$ is called
\emph{admissible}
if it is measurable, regarded as a function of two variables, on the
Lebesgue space $(X\times X, \mu \times \mu)$ and there exists a subset
$X_1\subset X$ of full measure such that the semimetric space
$(X_1,\rho)$ is separable.
\end{definition}

In other terms the separability condition is equivalent to
the requirement that measure $\mu$
is a Radon (or $\sigma$-compact) Borel measure w.r.t. (semi)metric $\rho$.

Since semimetrics play an essential role in our considerations, we use
basic notions of the theory of metrics in the case of semimetrics,
too. For example, speaking about the Borel $\sigma$-algebra of sets
in the case of a semimetric space, we mean the $\sigma$-algebra generated by
the open (in the sense of the semimetric in consideration) sets. Of course, this
$\sigma$-algebra does not in general separate points. One can easily
see that if $\rho$ is an admissible
metric (resp. semimatric) in a space $(X,\mu)$, then the measure $\mu$ is Borel with respect to
$\rho$, and the completion of appropriate subset $X_1\subset X$ of full measure
with respect to $\rho$ is a complete separable metric (= Polish) space (resp. complete separable semimatric
space) in which the
measure $\mu$ is nondegenerate (nonempty open sets have positive
measure).
%

An important class of admissible metrics is that of {\it
block semimetrics}. Let
$\xi$ be a partition of the space
$(X,\mu)$ into finitely or countably many measurable sets
$X_i$, $i=1,2, \dots$; the block semimetric $\rho_\xi$ corresponding
to $\xi$ is
defined as follows:
$\rho_{\xi}(x,y)=0$ if $x,y$ lie in the same set $X_i$ for some $i$, and
$\rho_{\xi}(x,y)=1$ otherwise. It is called a {\it cut
semimetric} (or just a {\it cut}) if
$\xi$ is a partition into two
subsets.

A triple $(X,\mu,\rho)$, where $(X,\mu)$ is a Lebesgue space and
$\rho$ is an admissible metric, will be called an
{\it admissible metric triple}, or, in short, an {\it admissible
triple}. In what follows, we are mostly interested in the case where
the measure $\mu$ is continuous (though we do not specify this
explicitly), but nevertheless all definitions make sense for an arbitrary
(in particular, finite) Lebesgue space. Unless otherwise stated, we
assume that an admissible metric $\rho$ is summable:
$$\int\limits_X\int\limits_X \rho(x,y)d\mu(x)d\mu(y)<\infty.$$
In other words, $\rho \in L^1(X \times X)$. However, some results
hold without this assumption; moreover, replacing the metric with an equivalent
one, we can arrive at the case of a summable metric.

Obviously, the (summable) admissible metrics form a cone in the space $L^1(X \times X,\mu \times \mu)$, which will
be denoted by $\mathcal{A}{\rm dm}(X,\mu)$.

The group $\mathfrak{G}$ of all automorphisms
(i.e., measurable, $\bmod 0$ invertible, $\mu$-preserving
transformations) of the space $(X,\mu)$ acts in
$L^1(X \times X,\mu \times \mu)$ in a natural way, and this action preserves the cone
$\mathcal{A}{\rm dm}(X,\mu)$ of admissible metrics.

As mentioned in the introduction, in what follows we fix a measure
and vary admissible metrics. It is useful to give a definition of an
admissible metric which is formally
less restrictive, but,  however,
turns out to be equivalent to the original one.

\begin{definition}
An \emph{almost metric} on a Lebesgue space $(X,\mu)$
is a measurable nonnegative function
$\rho$ on $(X\times X,\mu\times \mu)$ such that $\rho(x,y)=\rho(y,x)$
for almost all pairs of points $x, y \in X$
and $\rho(x,z)\leq \rho(x,y)+\rho(y,z)$ for almost
all triples of points $x, y, z \in X$.

An almost metric $\rho$ is called \emph{essentially separable} if for
every $\eps>0$, the space
$X$ can be covered by a countable family of measurable sets with
essential diameter (= essential supremum of the distances between
points) less than
$\eps$.
\end{definition}

In \cite{PZ}, the following correction theorem was proved.

\begin{theorem}\label{ispravlenie}
{\rm1)} Let $\rho$ be an almost (semi)metric on $X$. Then one can
modify it on a set of zero measure in $X$ so that the modified function
is an almost everywhere finite semimetric on $X$.

{\rm2)} Besides, if the almost
semimetric $\rho$ is essentially separable,
then the modified semimetric can be chosen so that the semimetric
space $(X,\rho)$ is separable and the corresponding triple is admissible.
\end{theorem}

Note that the limit in measure (or the almost everywhere limit) of a sequence of
(almost) metrics may turn out to be an almost metric, but the
correction theorem says that this limit is equivalent to a semimetric.
Thus in what follows we always assume that all almost metrics
obtained by limit procedures are corrected to semimetrics, that is,
the limit of a
sequence of semimetrics  with respect to almost everywhere
convergence is a semimetric or a metric. By the same theorem, the limit of a sequence of
semimetrics in the space $L^1$ can also be assumed to be a semimetric.

In what follows, it is convenient to use the following notation.

\begin{definition}
Let $A\subset X$, and let $\rho$ be a measurable semimetric on $X$. By
$\dm{\rho}{A}$ and $\esdm{\rho}{A}$ we denote the
diameter and the essential diameter of the set $A$ in the semimetric
$\rho$, respectively.
\end{definition}

\subsection{The entropy of metric measure spaces; equivalent
definitions of admissible metrics}

Now we introduce the notion of the $\eps$-entropy of a metric on a
measure space, which is heavily used in the sequel. The following
definition goes back to Kolmogorov.

\begin{definition}\label{kolm}
Let $(X, \rho)$ be a metric space equipped with a Borel probability
measure $\mu$. Consider the smallest positive integer $k$ for which
$X$ can be represented as the union of sets
$X_0,X_1, \dots, X_k$ such that $\mu(X_0)<\eps$
and $\dm{\rho}{X_j}<\eps$ for $j=1,\dots,k$.
The \emph{$\eps$-entropy} of the admissible triple
$(X, \mu, \rho)$ is
$$\Ent{\eps}{\rho,\mu}=\mathrm{log} k$$
 (the logarithm is binary). If such $k$ does not exist, we set
$\Ent{\eps}{\rho,\mu}=\infty$.
\end{definition}

However, it turned out that in some situations it is more convenient to use
another definition, which was suggested in
\cite{Mark} and involves the Kantorovich metric (or any other natural
metric) in the space of measures defined on a metric space.

\begin{definition}
Let $(X,\rho)$ be a separable metric space. The \emph{Kantorovich} (or
\emph{transportation}) metric $K_{\rho}$ on the simplex of Borel probability
measures on $X$ is defined by the formula
$$ K_{\rho}(\mu_1, \mu_2)=\inf_{\Psi}\bigg\{\iint_{X\times
X}\rho(x,y)d\Psi(x,y)\bigg\},$$
where $\Psi$ ranges over the set of all Borel probability measures on
$X\times X$ whose projections to the factors coincide with the measures
$\mu_1$ and $\mu_2$, respectively.
The $\eps$-entropy of an admissible triple $(X,\mu,\rho)$ is the
following function of $\eps$:
$$\EntK{\eps}{\rho,\mu}=\inf\{H(\nu)\colon K_{\rho}(\mu,\nu)<\eps\};$$
here $\nu$ ranges over all finite atomic measures on $X$ and the
entropy of an atomic measure is defined in the usual way:
$H(\sum_k c_k\delta_{x_k})=-\sum_k c_k\log c_k$.
\end{definition}

For a compact metric space, estimates  relating these two
definitions of the $\eps$-entropy are given in
\cite{Mark}. 

The following theorem contains a series of equivalent definitions of
admissible semimetrics, generalizing the results of
\cite{Mark,PZ}.

\begin{theorem}\label{kritdop}
Let $\rho$ be a measurable semimetric on $(X,\mu)$.
Then the following conditions are equivalent:
\begin{itemize}
\item[\rm1)] The triple $(X,\mu,\rho)$ is admissible, i.e., the
semimetric $\rho$ is admissible for the measure space
$(X,\mu)$.
\item[\rm2)] For every $\eps>0$, the semimetric $\rho$ has a finite
$\eps$-entropy: $\Ent{\eps}{\rho,\mu}<\infty$.
\item[\rm3)] The measure $\mu$ can be approximated in the metric
$K_{\rho}$ by discrete (=finitely supported) measures.
\item[\rm4)] For $\mu$-almost all $x\in X$ and every $\eps >0$,
the ball of radius $\eps$ (in the metric $\rho$) centered at $x$
has positive measure.
\item[\rm5)] For every $\eps>0$, the space $X$ can be represented as the
union of sets  $X_0,X_1, \dots, X_k$ such that $\mu(X_0)<\eps$ and
$\esdm{\rho}{X_j}<\eps$ for $j=1,\dots,k$.
\item[\rm6)] For every measurable set $A$ of positive measure, the
essential infimum of the function
$\rho$ on $A\times A$ is zero.
 \end{itemize}
\end{theorem}

Let us comment on some implications.

\begin{proof}
In \cite{PZ} it was proved that conditions~1),~2), and~4) are equivalent. The
equivalence of~2) and~5) is obvious, since if
$\esdm{\rho}{X_j}<\eps$, then $X_j$ can be partitioned
into two sets one of which has zero measure and the other one
has diameter at most $2\eps$. Really, if $\esdm{\rho}{X_j}<\eps$, then for almost every $x \in X_j$ for almost all
  $y \in X_j$ the inequality $\rho(x,y)<\eps$ holds. Fix some point $x_0 \in X_j$ such that $\mu(\{y\in X_j \colon \rho(x_0,y)\geq \eps\})=0$.
 Then, by triangle inequality, $\dm{\rho}{\{y\in X_j \colon \rho(x_0,y)< \eps\}}<2\eps$.

Now we prove that 2) implies 3).
Since for every $\eps>0$, the $\eps$-entropy  of $\rho$ (in the sense of Definition \ref{kolm})
is finite, there exists a partition of $X$ into sets
$X_0,X_1,\dots,X_k$ such that $\mu(X_0)<\eps$ and
$\dm{\rho}{X_j}<\eps$ for $j\geq 1$. For the set $X_0$, choose a point
$x_0\in X$, and for each of the sets $X_j$, $j\geq 1$, choose an
arbitrary point $x_j \in X_j$. Consider the atomic measure
$$\nu=\sum\limits_{j=0}^k \mu(X_j)\delta_{x_j}$$
and write the inequality
$$K_{\rho}(\mu,\nu)\leq \sum\limits_{j=0}^k \int\limits_{X_j}\rho(x_j,y)d\mu(y)\leq \eps+\int\limits_{X_0}\rho(x_0,y)d\mu(y).$$
Choosing $x_0$ appropriately, we can make the last term not bigger
than its mean value
$$\int\limits_{X}\int\limits_{X_0}\rho(x,y)d\mu(y)d\mu(x).$$
Thus, for an appropriate choice of $x_0$, we have
$$K_{\rho}(\mu,\nu)\leq \eps+\int\limits_{X_0}\int\limits_X\rho(x,y)d\mu(x)d\mu(y).$$
This estimate corresponds to transferring whole $X_i$ to $x_i$.
The latter expression is small for sufficiently small
$\eps$ by the absolute continuity of the integral and the summability of
the function $\rho$.

Next we prove that 3) implies 6).
If 6) does not hold, then there exist
$\eps>0$ and a set $A$ of positive measure such that
$\rho(x,y)\geq \eps$ for almost all pairs
$x,y \in A$. But then for every
$y\in X$, for almost all $x \in A$, we have $\rho(x,y)\geq
\eps/2$, so that $K_{\rho}(\mu,\nu)\geq \mu(A)\eps/2$ for every atomic measure
$\nu$, a contradiction with 3).

Finally, we prove that 6) implies 1), namely, we assume that $\rho$ is
not admissible and prove that~6) fails. For every
fixed  $\eps>0$, the function $x \to \mu(\{y\in X \colon \rho(x,y)<\eps\})$
is measurable by Fubini's theorem, so that the set
$A_{\eps}=\{x \colon \mu(\{y\in X \colon \rho(x,y)<\eps\})=0\}$
is measurable. If $\rho$ is not admissible, then~4)
fails, hence for some
$\eps>0$ the set $A_\eps$ has positive measure. Taking
$A=A_{\eps}$, we see that the essential infimum of $\rho$ on
$A\times A$ is positive.
\end{proof}

Two more definitions of admissible metrics are given in Section~2.6,
one in terms of averages of distances over sets of positive
measure, and the other one in terms of random distance matrices, which are
invariants of metric triples.

\subsection{The theorem on conditional metrics}

In this section, we prove a result similar to the well-known theorem
on the existence of conditional measures (``Rokhlin's canonical
system of measures'') for measurable partitions: a theorem on the
existence of a system of conditional admissible metrics on almost all
elements of a partition. Thus we will show that if
$(X,\mu,\rho)$ is an admissible metric triple, then for
every measurable partition $\xi$
of $X$, almost all elements of $\xi$ can be equipped with a
canonical structure of a metric triple with
respect to the induced metric. A
nontrivial issue is to define metrics on the elements of the
partition.

Recall that a measurable partition $\xi$ of a Lebesgue space can be
defined as the partition into the inverse images of points under a
measurable map from
$(X,\mu)$ to another Lebesgue space, e.g., under a
measurable real-valued function or vector-valued function with values in a separable vector
topological space. An intrinsic definition of a measurable partition
suggested in \cite{Ro} relies on the existence of a countable basis of
measurable sets determining the partition. For a measurable partition
$\xi$, the quotient space $X/\xi$
 (the base of $\xi$) is a Lebesgue
space (sometimes, by definition); the image of
$\mu$ under the canonical quotient map
$\pi:X \rightarrow X/\xi$ is a measure
$\mu_{\xi}$ on  $X/\xi$. The main characteristic property of a
measurable partition is the existence and uniqueness of a canonical
system of conditional measures
$\mu^C$ on $\mu_\xi$-almost all
elements $C \in X/{\xi}$ of $\xi$, the spaces
$(C,\mu^C)$ being Lebesgue spaces. In fact, the theorem on the
existence of conditional measures is a
theorem on an integral representation of the projection in
$L^2$ to the subspace of functions which are constant on the elements
of partition $\xi$, or, in other words, this is an integral representation of
the operator of the conditional expectation operator. The crucial fact is that for
every $\mu$-measurable map $f$ with values in a space $V$ with a Borel structure
(e.g., a measurable real-valued function), and for almost all elements
$C\in X/{\xi}$ of $\xi$, the restriction $f|_C$ of $f$ to $C$ is
measurable with respect to the conditional measure
$\mu^C$, and the map $C\mapsto f|_C$
is measurable on the base of $\xi$.
For a summable function $f$, this means that an analog of Fubini's theorem
holds: the integral of $f$ over the whole space is equal to the iterated
integral computed first over the elements and then over the base.
All these definitions are well-behaved with respect to modifying a measurable
partition on a set of zero measure.

Below we will obtain a similar result for measurable partitions of
measure spaces equipped with a metric. Consider an admissible triple
$(X,\mu,\rho)$ and assume that in $X$ we are given a
measurable partition $\xi$. Denote by  $\mu_{\xi}$ the quotient measure on the
quotient space $X/\xi$, i.e., on the base
 of $\xi$. We will regard
elements (fibers) of $\xi$ either as points of the base, denoting them
by $C \in X/\xi$, or, if  convenient, as subsets of
$X$, writing $C\in \xi$. The conditional measure on an element $C$
will be denoted by $\mu^C$.

Using this notation, we state the theorem on the existence of
conditional metrics on almost all elements of a measurable partition, and
measurability of the dependence of a metric as a function of element $C$
of the partition in appropriate sense. For making this  statement rigorous
we use a metric invariant of a function of two variables (in particular, on a metric)
on a measure space ---  so called {\it matrix distributions} which was introduced in \cite{Clas}.
This notion gives a simple way to define what does it mean measurability of the family of metrics,
which are defined on the various spaces (on the elements of a partition) --- see item 2 in the theorem.
\begin{theorem}\label{cond}
\begin{itemize}
\item[\rm1)]  The restriction of the metric $\rho$ to $\mu_{\xi}$-almost
every element $C \in \xi$ of the
partition $\xi$ is well defined and determines the structure of an
admissible triple $(C,\mu^C,\rho^C)$ for almost all
$C\in \xi$.

\item[\rm2)] Let $n$ be a positive integer, let
$\Omega$ be any open set in the $n^2$-dimensional
space of $n\times n$ matrices. For almost any element $C$ of $\xi$
one may define by 1) an admissible triple $(C,\mu^C,\rho^C)$.
Let $p_{\Omega}(C)$ denotes the probability that a matrix
$(\rho(z_i,z_j))_{1\leq i,j\leq n}$ belongs to $\Omega$,
where $z_1,\dots,z_n$ are independent points in $C$ distributed
by $\mu^C$. Then $p_{\Omega}$ is a measurable function of $C$.

\end{itemize}
\end{theorem}

It may seem that in order to obtain the required assertions,
it suffices to restrict the metric to almost every
element of the partition, but
this is not so. The problem is that for measurable functions of two
(or several) variables, e.g., for an admissible metric, one
cannot directly use a Fubini-like theorem on the measurability of
restrictions  of functions to the
elements of the partition.
Moreover, in general this is not true for
an arbitrary function. Indeed, the set of pairs
$(x,y)$ lying in the same fiber of $\xi$ has (in general)
zero measure in $X\times X$, hence there is no known
canonical way to
restrict an arbitrary $\mu^2$-measurable
function $f(x,y)$ to this set.

Hence, in order to prove that
the metrics on the elements are admissible
and measurable over the base of the
partition, one should use special properties of
these functions.
It turns out that the needed property is admissibility.
Note that similar questions, in spite of their importance, have not
yet been studied in general setting. We use the separability of an
admissible metric, which ensures that this
metric can be defined by a vector function
of {\it one variable}. The trick of passing to a sequence
for one or both arguments of a function of
two arguments, mentioned above
and exploited below, was essentially used in
\cite{Clas} for the classification of measurable functions of several
variables via a {\it random choice} of sequences.

\begin{proof}
Choose a sequence $x_1,x_2,\dots$ in $X$, which is dense
in some subset $X_1$ of full measure in $X$.
We use the functions
$f_n(\cdot)=\rho(\cdot,x_n)$, $n=1,2, \dots$. We also require
that those functions are simultaneously measurable on $X_1$.
Further, note that since the sequence
$\{x_n\}$ is dense, we have
$$ \rho(x,y)= \inf_n \{f_n(x)+f_n(y)\}.$$
Therefore, for almost all elements $C$ of
$\xi$ equipped with the conditional measures $\mu^C$, this formula defines a metric
as a measurable function of two variables. The admissibility of the
triple $(C,\mu^C,\rho^C)$ is straightforward, because a subspace of a
separable metric space is separable. The fact that
$\rho^C$ is summable with respect to the measure $\mu^C\times \mu^C$ for almost
every $C$ easily follows from the triangle inequality and the
separability.

Now we should explain measurability statement 2). Without loss of generality,
$\Omega$ is a cylinder $\{(a_{i,j})_{1\leq i,j\leq n}: 0\leq a_{i,j}<p_{i,j} \}$
for fixed positive numbers $p_{i,j}$. The condition
$\rho(x,y)= \inf_n \{f_n(x)+f_n(y)\}<p$ is equivalent to the countable number of conditions like
$f_n(x)<r_1,f_n(y)<r_2$ for some index $n$ and rationals $r_1,r_2$ with $r_1+r_2<p$.
So, the probability that a random distance matrix belongs to $\Omega$
may be expressed via probabilities that $\rho(z_i,x_n)$ belongs to some
interval on a real line. Such events are (at last) independent, and the product of corresponding
probabilities is measurable, since each of them is measurable by Rokhlin theorem.
\end{proof}

\subsection{The space of admissible metrics. The definition and
properties of the m-norm}

When working with admissible semimetrics, it is convenient to introduce a
special norm on the cone of admissible metrics
$\mathcal{A}{\rm dm}$, which we call the
m-norm; it is defined on $\mathcal{A}{\rm dm}$ and on a wider vector subspace of
$L^1(X^2)$.

\begin{definition}
Given a function $f\in L^1(X^2)$, we define a finite or infinite
norm of $f$ as
$$\mnorm{f}=\inf\{\Lnorm{\rho}\colon \rho \mbox{ is a semimetric, }
\rho(x,y)\geq |f(x,y)| \mbox{ for almost all } x,y \in X\}.$$
\end{definition}

Note that $\mnorm{\cdot}$ is indeed a norm, in the sense that it is homogeneous
and satisfies the triangle inequality. If $f$ is a
semimetric, then
$\mnorm{f}=\Lnorm{f}$. It follows directly from the definition that
for every $f$ we have
$\mnorm{f}\geq \Lnorm{f}$. Hence convergence in the m-norm
implies convergence in
$L^1(X^2)$. In the theory of partially ordered Banach spaces, such a
convergence is called convergence with a regulator.
Note that the operators
corresponding to measure-preserving automorphisms preserve also the
m-norm.

Consider the set of all functions in
$L^1(X^2)$ with finite m-norm:
$$\mathbb{M}=\{f\in L^1(X^2)\colon \mnorm{f} < \infty \}.$$
Clearly, $\mathbb{M}$ is a linear subspace in $L^1(X^2)$.

\begin{lemma}
The space $\mathbb{M}$ is complete in the m-norm.
\end{lemma}
\begin{proof}
Let $f_n$ be a Cauchy sequence with respect to the m-norm. We
will show that it has a limit in the m-norm. Since the $L^1$ norm is
dominated by the m-norm, $f_n$ is also a Cauchy sequence in
$L^1(X^2)$, so that it has a limit $f\in L^1(X^2)$. Thinning the
sequence, we may assume that
$f_n$ converges to $f$ almost everywhere and, besides,
$\mnorm{f_n-f_{n+1}}< \frac{1}{2^n}$  for all $n$. By the
definition of the m-norm, this means that there exists a semimetric
$\rho_n$ that dominates $|f_n-f_{n+1}|$ almost everywhere and satisfies
$\Lnorm{\rho_n}< \frac{1}{2^n}$. Note that the semimetric
$\sum_{k=n}^{\infty}\rho_k$ dominates the difference
$|f_n-f|$ almost everywhere, so that $\mnorm{f_n-f}\leq
\frac{1}{2^{n-1}}$. It follows that the sequence
$f_n$ converges to $f$ in the m-norm, as required.
\end{proof}

Now we will study simple properties of convergence of semimetrics.

\begin{lemma}\label{shod}
If a sequence of semimetrics
$\rho_n$ converges to a function $\rho$ in the m-norm, and for every
$\eps>0$ the entropy
 $\Ent{\eps}{\rho_n,\mu}$ is finite for all sufficiently large $n$, then
$\rho$ is an admissible semimetric.
\end{lemma}

\begin{corollary}\label{c1}
If a sequence of admissible semimetrics
$\rho_n$ converges to a function $\rho$ in the m-norm, then
$\rho$ is also an admissible semimetric.
\end{corollary}

\begin{proof}[Proof of Lemma~\ref{shod}]
Since the sequence $\rho_n$ converges in the m-norm, it also converges
in the space $L^1(X^2)$, so that we may assume that the limit function
$\rho$ is a semimetric. It remains to prove that $\rho$ is admissible.
For this we will show that its $\eps$-entropy is finite  for every
$\eps$. First we prove an auxiliary
proposition.

 \begin{statement}
 \label{prop1}
If $p$ is a measurable semimetric on
$(Y,\mu)$ such that $\|p\|_{_{L^1(Y^2)}}<\frac{\eps^2}{2}$
then there exist two disjoint sets
$Y_0, Y_1$ with $Y_0\cup Y_1=Y$ such that
$\mu(Y_0)\leq \eps$ and $\dm{p}{Y_1}\leq\eps$.
 \end{statement}
 \begin{proof}
Note that the map
$x\to \mu(\{y\in Y\colon p(x,y)\geq \eps/2\})$ is measurable by
Fubini's theorem, and its integral over $Y$ is
bounded from above by
$\frac{\frac{\eps^2}{2}}{\frac{\eps}{2}}=\eps$  by Chebyshev's
inequality. Hence we can choose
$x_0$ such that the measure of the set
$Y_0=\{y\in Y\colon p(x_0,y)\geq \eps/2\}$ does not exceed $\eps$. But
for any $x,y \in Y_1=Y\setminus Y_0$, the triangle inequality implies
that $p(x,y)\leq p(x,x_0)+p(y,x_0)\leq \eps$. The proposition follows.
 \end{proof}

Returning to the proof of the lemma, we fix
$\eps>0$ and prove that
$\Ent{4\eps}{\rho}$ is finite. For large $n$,
we have $\mnorm{\rho_n-\rho}<\eps^2/2$. By the definition of the
m-norm, this means that there exists a semimetric $p$ such that
$\Lnorm{p}<\eps^2/2$ and $\rho \leq p+ \rho_n$ almost everywhere. As
we have just proved, the set $X$ can be partitioned into two sets
$X_0$ and $X_1$ such that $\mu(X_0)\leq \eps$ and $p(x,y)\leq \eps$
for all $x,y \in X_1$. Choosing $n$ large enough, we may assume that
the number $\Ent{\eps}{\rho_n}$ is finite, i.e., we can find a partition
$X=A_0\cup A_1 \cup \dots \cup A_k$ such that $\mu(A_0)<\eps$ and
$\dm{\rho_n}{A_j}< \eps$ for $j\geq 1$.

Now we construct a partition for the semimetric
$\rho$ as follows. Put
$B_0=A_0 \cup X_0$ and $B_j=A_j \cap X_1$ for $j=1,\dots,k$. Clearly,
$\mu(B_0)\leq \mu(A_0)+\mu(X_0)< 2\eps$. For every $j>0$, for almost all
$x,y\in B_j$, we have the inequality
$\rho(x,y)\leq \rho_n(x,y)+p(x,y)\leq \eps +\eps=2\eps$, which shows
that $\esdm{\rho}{B_j}\leq 2\eps$. Thus we have shown that for every
$\eps>0$ the number $\Ent{4\eps}{\rho}$ is finite and, consequently,
that the semimetric $\rho$ is admissible.
\end{proof}

The following simple lemma says that the limit of a sequence of
``uniformly bounded'' admissible semimetrics in the space
$L^1(X^2)$ is again an admissible semimetric. The boundedness here is
understood in the entropy sense.

\begin{lemma}\label{lem3}
If $M$ is a set of admissible semimetrics such that
the set $\{\Ent{\eps}{\rho,\mu}\colon\rho \in M\}$ is bounded for every
$\eps>0$,
then the closure of $M$ in the space
$L^1(X^2)$ consists of admissible semimetrics only.
\end{lemma}
\begin{proof}
Take an arbitrary function
$\rho$ from the closure of $M$ in
$L^1$. We will prove that it is an admissible semimetric. We know
that there exists a sequence of semimetrics
$\{\rho_n\}\subset M$ that converges to $\rho$ in $L^1(X^2)$. Clearly,
$\rho$ is a semimetric, and one should only check that it is admissible.

Assume to the contrary that $\rho$ is not admissible. Then, by
Theorem~\ref{kritdop}, there exist $\eps>0$ and a set $A\subset X$ of
positive measure such that $\rho(x,y)\geq \eps$ for almost all $x,y \in A$.
Decreasing $\eps$ if necessary, we may assume that
$\mu(A)\geq \eps$.

Using the boundedness of entropies for
$\eps/2$, for each of the semimetrics $\rho_n$ we find a partition
$X=X_0\cup X_1\cup \dots \cup X_k$ such that $\dm{\rho_n}{X_i}\leq \eps/2$
for all $i=1,{\ldots} ,k$ and $\mu(X_0)\leq \eps/2$. Of course, this partition
may depend on $n$, but the number $k$ can be chosen
to be universal, since the entropies are bounded. Note that at least
one of the sets
$(X_i\cap A)$, $i\geq 1$, has measure not less than $\frac{\eps}{2k}$.
Moreover, for almost all
$x,y\in(X_i\cap A)$, we have
$$\rho(x,y)-\rho_n(x,y)\geq \eps - \frac{\eps}{2}= \frac{\eps}{2},$$
whence
$$\Lnorm{\rho-\rho_n}\geq (\frac{\eps}{2k})^2 \frac{\eps}{2}>0.$$
The latter inequality contradicts the convergence of
$\rho_n$ to $\rho$ in $L^1(X^2)$, and the lemma follows.
\end{proof}

In conclusion of this section, we prove a lemma on pointwise
convergence of admissible semimetrics.

\begin{lemma}
\label{point}
Assume that a sequence of admissible semimetrics
$\rho_n$ converges to an admissible semimetric
$\rho_{\rm lim}$ almost everywhere with respect to the measure
$\mu\times \mu$. Then there exists a set $X'\subset X$ of full measure
such that for any $x,y\in X'$,
$$\limsup\limits_n \rho_n(x,y)=\rho_{\rm lim}(x,y).$$
Besides, if $x,y\in X'$ and $\rho_{\rm lim}(x,y)=0$, then
$$\lim\limits_n \rho_n(x,y)=0.$$
\end{lemma}
\begin{proof}
Consider the function $\bar{\rho}(x,y)=\limsup\limits_{n}
\rho_n(x,y)$. The functions $\bar{\rho}$ and $\rho_{\rm lim}$ coincide on
a set of full measure in $X^2$. We must prove that they coincide on
the square of a set $X'$ of full measure in $X$. Note that the function
$\bar{\rho}$ satisfies the triangle inequality everywhere (as upper limit
of semimetrics); also it is finite almost
everywhere with respect to the measure
$\mu^2$, because the function $\rho_{\rm lim}$ is finite a.e. Put
$X''=\{x \in X \colon \mu(\{y \colon \bar{\rho}(x,y)=+\infty\})=0\}$. Note
that $\mu(X'')=1$. We will prove that $\bar{\rho}(x,y)<+\infty$ for any $x,y \in X''$. Indeed, if
$\bar{\rho}(x,y)=+\infty$, then for every $z \in X$ we have either
$\bar{\rho}(x,z)=+\infty$ or $\bar{\rho}(y,z)=+\infty$, contradicting
the choice of $X''$. Thus on  $X''$ the semimetric $\bar{\rho}$
is finite and coincides almost everywhere with
$\rho_{\rm lim}$. Using the characterization of admissibility in terms of
the measures of balls from Theorem~\ref{kritdop} for the semimetrics
$\rho_{\rm lim}$ and $\bar{\rho}$, we see that
$\bar{\rho}$ is also admissible. Then, by \cite[Theorem~3]{PZ}, there
exists a set $X'\subset X''$ of full measure such that
$\rho_{\rm lim}=\bar{\rho}$ on the square of $X'$.

The last claim is obvious.
\end{proof}

\subsection{Convergence of admissible metrics. A precompactness criterion}

\begin{lemma}
\label{blim}
Assume that a sequence of uniformly bounded semimetrics
$\rho_n$ converges to an admissible semimetric
$\rho$ in $L^1$. Then this sequence converges in the
m-norm to the same limit.
\end{lemma}

\begin{proof}
Let $R$ be a constant bounding all semimetrics
$\rho_n, \rho$. Fix $\eps>0$ and, using the admissibility of
$\rho$, find a partition of the space $X$ into sets
$A_0, A_1,\dots,A_k$ such that $\mu(A_0)<\eps$ and $\dm{\rho}{A_j}\leq
\eps^2$ for $j>0$.
We may assume that
 $$\delta=\min\{\mu(A_j) \colon j=1,{\ldots} ,k\}>0.$$
Note that for every $j>0$ the sequence of restricted semimetrics
$\rho_n|_{_{A_j^2}}$ converges to the semimetric $\rho|_{_{A_j^2}}$
in the space $L^1(A_j^2)$. By construction, the limit semimetric does not exceed
$\eps^2$ everywhere on $A_j$, hence for sufficiently large $n$ we have
 $$\|\rho_n|_{_{A_j^2}}\|_{_{L^1(A_j^2)}}\leq 2\eps^2 \mu(A_j)^2.$$

Now consider the set $A_j$ equipped with the normalized measure $\mu/\mu(A_j)$
and apply Proposition~\ref{prop1} to the restriction of the semimetric
$\rho_n$ to $A_j$. We see that
$A_j$ can be partitioned into two sets $B_j(n), C_j(n)$ such that
$\mu(B_j(n))\leq 2\eps \mu(A_j)$ and $\dm{\rho_n}{C_j(n)}\leq 2\eps$.
This immediately implies that $\mu(C_j(n))\geq (1-2\eps)\mu(A_j)$.

Choose $n$ so large that these inequalities hold for all
$j=1,{\ldots} ,k$. We put $C(n)=\bigcup\limits_{j=1}^k C_j(n)$ and
prove that if $n$ is sufficiently large, then
$|\rho_n(u,v)-\rho(u,v)|\leq 10 \eps$ for any
$u,v\in C(n)$. If $u,v\in C_j(n)$ for some $j$, then
$|\rho_n(u,v)-\rho(u,v)|\leq \rho_n(u,v)+\rho(u,v)\leq 3\eps$  by construction. Now let
$u_0 \in C_i(n)$, $v_0 \in C_j(n)$, and $i\ne j$. If
$|\rho_n(u_0,v_0)-\rho(u_0,v_0)|>10\eps$, then for all
$u\in C_i(n)$, $v \in C_j(n)$ we have
\begin{eqnarray*}
|\rho_n(u,v)-\rho(u,v)|\geq |\rho_n(u_0,v_0)-\rho(u_0,v_0)|\\- (\rho_n(u_0,u)+\rho_n(v_0,v)+\rho(u_0,u)+\rho(v_0,v))>4\eps.
\end{eqnarray*}
But then
$\Lnorm{\rho_n-\rho}\geq 4\eps \mu(C_i)\mu(C_j)\geq 4 \eps (1-2\eps)^2\delta^2$,
which cannot be true for large $n$. Thus for all sufficiently large
$n$, for any two points
$u,v \in C(n)$ we have $|\rho_n(u,v)-\rho(u,v)|\leq 10 \eps$. It follows
from the construction that the measure of
$C(n)$ is large, more exactly,
$\mu(C(n))\geq (1-2\eps)(1-\eps)$.

Define a metric $p_n$ as follows. On the set
$C(n)\times C(n)$ it is identically equal to $10\eps$, and on the
remaining set it is equal to
$2R+10\eps$. We have just proved that on
$C(n)$ this metric dominates the difference
$|\rho_n-\rho|$. On the remaining set, it also dominates the
distance, because all original semimetrics are bounded by $R$. Since
$R$ is fixed, for sufficiently small
$\eps$ the metric $p_n$ has an arbitrarily small $L^1$ norm. Thus the
sequence $\rho_n$ converges to $\rho$ in the m-norm, and the lemma
follows.
\end{proof}

In what follows, we need a lemma on cut-offs of semimetrics.

Given an arbitrary function $f$ and a real number $R$, denote by $f^R$
the cut-off of $f$ of level $R$, that is,
$f^R(\cdot)=\min(f(\cdot),R)$.

\begin{lemma}
\label{cut}
For a summable semimetric $p$ on the space
$(X, \mu)$ and every $R>0$,
$$\mnorm{p-p^{2R}} \leq 2 \int\limits_{p>R}p d\mu^2.$$
\end{lemma}

\begin{proof}
Choosing an arbitrary point $x \in X$, consider the ball
$B=\{y\in X\colon p(x,y)\leq R\}$ and its complement $A=X\setminus B$.
Now we define a semimetric $q$ as follows:
$$
q(u,v)=
\begin{cases}
0,& u,v \in B,\\
p(u,v),& u,v \in A,\\
p(u,x), & u \in A, v \in B, \\
p(v,x), & u \in B, v \in A. \\
\end{cases}
$$
One can easily check that $q$ is indeed a semimetric and, besides,
for any $u,v \in X$ we have $p(u,v)-p^{2R}(u,v)\leq q(u,v)$. Thus, by
the definition of the m-norm,
\begin{eqnarray*}
\mnorm{p-p^{2R}}&\leq& \Lnorm{q}=\int\limits_{X\times X} q
d\mu^2=\big(\int\limits_{A\times A}+\int\limits_{A\times
B}+\int\limits_{B\times A}+\int\limits_{B\times B}\big) q d\mu^2\\
&=&\int\limits_{A\times A} p(u,v) d\mu(u)d\mu(v)+2 \int\limits_{A\times
B} p(u,x) d\mu(u)d\mu(v)\\
&\leq&\int\limits_{A\times A}( p(u,x)+p(x,v)) d\mu(u)d\mu(v)+ 2
\mu(B)\int\limits_{A} p(u,x) d\mu(u)\\
&=&2(\mu(A)+\mu(B))\int\limits_{A} p(u,x) d\mu(u)=2\int\limits_{A} p(u,x) d\mu(u).
\end{eqnarray*}
Now we can optimize this bound by choosing $x$. Note that the average
of the right-hand side  over
$x\in X$  coincides
with
$2\int\limits_{p>R}p d\mu^2$; hence, choosing $x$ appropriately, we
obtain the desired bound.
\end{proof}

We use this lemma to deduce a more general theorem.

\begin{theorem}\label{lem4}
Assume that a sequence of semimetrics
$\rho_n$ converges to an admissible semimetric $\rho$ in the space $L^1$.
Then this sequence converges in the m-norm to the same limit.
\end{theorem}
\begin{proof}
We just use the two lemmas already proved. Fix
$\delta>0$ and, using the absolute continuity of the integral of
$\rho$, choose $R>0$ so large that
 $$\int\limits_{\rho>R/2}\rho d\mu^2 < \delta.$$
Since the sequence  $\rho_n$ converges to $\rho$ in $L^1(X^2)$,
for sufficiently large $n$ we have
 $$\int\limits_{\rho_n>R}\rho_n d\mu^2 < 2\delta.$$
The cut-offs $\rho_n^{2R}$ converge to $\rho^{2R}$ in the
space $L^1(X^2)$, since for any functions $f,g$ we have
$$\Lnorm{f^{2R}-g^{2R}}\leq \Lnorm{f-g}.$$
Applying Lemma~\ref{blim} to the cut-offs, we see that for sufficiently
large $n$,
$$\mnorm{\rho_n^{2R}-\rho^{2R}}\leq \delta.$$
Using Lemma~\ref{cut} twice, we can write the inequality
 $$\mnorm{\rho_n-\rho}\leq \mnorm{\rho_n-\rho_n^{2R}}+\mnorm{\rho_n^{2R}-\rho^{2R}}+\mnorm{\rho-\rho^{2R}}\leq 4\delta.$$
Thus the sequence $\rho_n$ converges to $\rho$ in the m-norm, as
required.
\end{proof}

This theorem easily implies the following corollary.

\begin{corollary}\label{corcomp}
A set of admissible semimetrics is compact in the m-norm if and only
if it is compact in $L^1$.
\end{corollary}

In the remaining part of this section we prove a precompactness
criterion for the m-norm.

\begin{theorem}\label{crit}
Let $M$ be a set of admissible semimetrics on
$(X,\mu)$. Then $M$ is precompact in the m-norm if and only if the
following two conditions hold:
\begin{itemize}
\item[{\rm1)}] {\rm(uniform integrability)} the set $M$ is
uniformly integrable on $X^2$;
\item[{\rm2)}] {\rm(uniform admissibility)} for every
$\eps>0$ there exists a partition of $X$ into finitely many sets
$X_1,\dots,X_k$ such that for every semimetric $\rho\in M$ there
exists a set $A\subset X$ of measure less that $\eps$ such that
$\dm{\rho}{X_j\setminus A}< \eps$.
\end{itemize}
\end{theorem}

Note that condition~2)  in the statement of the theorem can be
replaced with the equivalent condition~2') in which
${\rm diam}$ is replaced by ${\rm essdiam}$.
Moreover, each of these conditions
implies that the set $\{\Ent{\eps}{\rho}\colon\rho \in M\}$
is bounded for every
$\eps>0$.

It is worth mentioning that we will use not only the definition of
uniform integrability, but also its reformulation. We will say that
a family of functions
$K \subset L^1(\Omega, \nu)$ is uniformly integrable if for every
$\eps>0$ there exists $\delta>0$ such that for every set
$A\subset \Omega$ with $\nu(A)<\delta$, for every function $f\in K$,
$$\int\limits_A |f| d\nu < \eps.$$

Now we proceed to the proof of the theorem.

\begin{proof}
First we will prove that if $M$ is precompact in the m-norm, then
conditions~1) and~2) are satisfied. Note that since the m-norm
dominates the $L^1$ norm, the set $M$ is precompact in the space
$L^1$ and hence uniformly integrable.

Consider an arbitrary finite partition $\xi$ of the space $X$ into sets
$X_1, \dots, X_k$. Assume that for some semimetric
$\rho$ the partition $\xi$ is an $\eps$-partition, i.e., there exists
an exceptional set $A$ such that
$\mu(A)<\eps$ and $\dm{\rho}{X_j\setminus A}<\eps$, $j=1,\dots, k$.
Using Proposition~\ref{prop1}, one can easily see that there exists
$\delta>0$ such that if $\mnorm{\rho-\rho_1}<\delta$, then $\xi$ is an
$\eps$-partition for
$\rho_1$, too. That is, the set of semimetrics for which a given
partition is an $\eps$-partition is open in the m-norm. We will refer to this set as corresponding to $\xi$. By the
Corollary~\ref{c1}, the closure of the set $M$ in the
m-norm consists only of admissible semimetrics, each having a
finite $\eps$-partition. Let us cover the closure of $M$ (which is a compact
set) by the open sets corresponding to finite partitions. This open
cover has a finite subcover. Clearly, the intersection of the
corresponding partitions is a universal $\eps$-partition for all
semimetrics in $M$, i.e., condition~2) is satisfied.

Now we will prove that conditions~1) and~2) are sufficient for $M$ to
be precompact.

First we prove that the set $M$ is precompact in the space
$L^1(X^2)$. It suffices to find, for every
$\eps>0$, a finite $4\eps$-net in the $L^1$-norm.

The uniform integrability of the family $M$ means that
$$\lim\limits_{R\to +\infty} \sup\limits_{\rho \in M} \iint_{\rho>R}\rho d(\mu\times\mu) =0.$$
Hence for sufficiently large $R$, all cut-offs of the functions are
close in $L^1$ (and even in the m-norm) to the corresponding
semimetrics from $M$. Therefore, it suffices to search for an
$\eps$-net in the set of cut-off semimetrics. For sufficiently large
$R$, we have
$$\mnorm{\rho^R-\rho}<\eps$$
 for every $\rho \in M$.
Note that the universal partition from condition~2) remains universal
also for all cut-offs $\rho^R$. The set of cut-off semimetrics will
be denoted by $M^R$.

Fix a small number $\delta>0$ which will be specified later, and, using
condition~2), find a universal
$\delta$-partition $X=X_1\cup \dots \cup X_k$. For every function
$\rho \in M^R$, find an exceptional set $A$ of measure at most
$\delta$ such that $\dm{\rho}{X_j\setminus A}<\delta$ for $j=1,\dots,k$.
Put $Y_j=X_j\setminus A$ and define a function $\bar{\rho}\in L^1(X^2)$
on each of the sets $X_i\times X_j$ as the average of $\rho$ over the set
$Y_i\times Y_j$. In the case where one of the sets
$Y_j$ has zero measure, we set the value
of $\bar{\rho}$ on this set equal to zero. We will prove that
$\bar{\rho}$ is close to $\rho$ in the
space $L^1(X^2)$. First,
both functions are bounded by $R$. Second, for any
$u_1,u_2 \in Y_i$, $v_1,v_2 \in Y_j$, we have the obvious inequality
$$|\rho(u_1,v_1)-\rho(u_2,v_2)|\leq \rho(u_1,u_2)+\rho(v_1,v_2)<2\delta.$$
Hence
$$\int\limits_{Y_i}\int\limits_{Y_j}|\rho-\bar{\rho}|d\mu^2 < 2\delta \mu(Y_i)\mu(Y_j).$$
The union of all sets of the form
$Y_i \times Y_j$ is exactly $(X \setminus A)^2$, whence
$$\int\limits_{X}\int\limits_{X}|\rho-\bar{\rho}|d\mu^2 < 2\delta \mu(X\setminus A)^2+2R \mu(A)<2\delta(1+R),$$
which is small for sufficiently small $\delta$. Thus we can
approximate every function from
$M^R$ by the corresponding function $\bar{\rho}$ with accuracy
$\eps/2$. But the set of all such functions
$\bar{\rho}$ is bounded in $L^1(X^2)$ and is contained in a
finite-dimensional subspace, so that it has a finite
$\eps/2$-net. It follows that in $M$ we can find a finite
$4\eps$-net with respect to the norm of the space $L^1(X^2)$.

Thus $M$ is precompact in
$L^1(X^2)$. Consider its closure $\bar{M}$ in $L^1(X^2)$.
By Lemma~\ref{lem3} (the condition of this lemma holds because of the uniform admissibility), all functions from $\bar{M}$ are admissible
semimetrics. Thus the set $\bar{M}$, which is compact in
$L^1$, consists of admissible semimetrics
only, so that, by Corollary~\ref{corcomp} of Lemma~\ref{lem4},
it is compact in the m-norm. Hence the set $M$ is precompact in the
m-norm.
\end{proof}

Theorems~\ref{lem4}, \ref{crit} and Lemma~\ref{lem3} easily imply the
following corollary.

\begin{corollary}\label{cor2}
If $M$ is a precompact set in the m-norm that consists of admissible
semimetrics, then its closures in
$L^1(X^2)$ and in the m-norm coincide and consist of admissible
semimetrics only. Also, $\eps$-entropies of semimetrics
in $M$ are uniformly bounded for any fixed $\eps>0$. In particular,
this holds for a sequence of admissible semimetrics,
converging in m-norm
(and hence by Lemma \ref{lem4}
for a sequence of admissible semimetrics, converging
in $L^1$ to admissible semimetric.)
\end{corollary}

The following criterion of precompactness
deals with convex sets of metrics. It is suggested
by applications in ergodic theory.

\begin{theorem}\label{kriko}
Let $M$ be a uniformly integrable convex family of admissible
semimetrics in the space
$\mathbb{M}$. Then $M$ is precompact in the m-norm if and
only if the $\eps$-entropies of
semimetrics in $M$ are uniformly (with respect to semimetric) bounded for every fixed
$\eps>0$.
\end{theorem}

\begin{proof}
The precompactness of $M$ implies the uniform boundedness of the
$\eps$-entropies, e.g., by item 2 in Theorem~\ref{crit}.

Now we prove that if the
$\eps$-entropies are uniformly bounded for every $\eps>0$, then $M$
is precompact in
$L^1(X^2)$. This will imply that $M$ is precompact in the
m-norm. Indeed, a set is precompact if and only if
every sequence of elements of this set has a Cauchy subsequence.
Thus if $M$ is precompact in  $L^1$, then every sequence of
elements of $M$ has a Cauchy subsequence, which converges to a
semimetric $\rho$ in $L^1$; since the
$\eps$-entropies are uniformly bounded, it follows from
Lemma~\ref{shod} that this semimetric is admissible.
Then, by Theorem~\ref{lem4}, the sequence converges to $\rho$ also in the space
$\mathbb{M}$.

Assume that $M$ is not precompact in
$L^1$. Then, for some $c>0$, we can choose a sequence of semimetrics
$\rho_1,\rho_2,\dots$  in
$M$ such that $\|\rho_i-\rho_j\|_{L_1}>c$
for all indices $1\leq i<j<\infty$. For the moment, fix
$\eps>0$ whose value will be specified later. Find a positive integer $k$ such
that for every metric
$\rho\in M$ there exists a partition of $X$ into sets
$X_0,X_1,\dots,X_k$ such that $\mu(X_0)<\eps$ and $|\rho(x,y)|< \eps$
for all $x,y\in X_i$, $i=1,2,\dots,k$.

Consider the semimetric
$\rho=\frac{\rho_1+\dots+\rho_n}n$; by convexity,
$\rho\in M$. The value of $n$ will also be specified later.
Consider the corresponding partition of $X$ into sets
$X_0,X_1,\dots,X_k$. Choose points $p_i$ in
$X_i$ arbitrarily for $i=1,\dots,k$.
For $s=1,2,\dots,n$, consider the function $d_s$ on  $X\times X$ defined as
$$
d_s(x,y)=\begin{cases}
       0,& x\in X_0\, \text{ or }\, y\in X_0,\\
       \rho_s(p_i,p_j),& x\in X_i,y\in X_j\quad (1\leq i,j\leq k).
       \end{cases}
$$
We will estimate the sum of the
$L^1$-distances between the pairs of functions
$d_s,\rho_s$ on $X\times X$. The measure of the set
$X_0\times X\cup X\times X_0$ is less than $2\eps$; the integral over
this set of each of the functions
$\rho_s$ does not exceed some value
$\delta(\eps)$ which is small provided that $\eps$ is small (this is
the uniform integrability of $M$). On
$X_i\times X_j$ we have
$$
|\rho_s(x,y)-d_s(x,y)|=|\rho_s(x,y)-\rho_s(p_i,p_j)|\leq
\rho_s(x,p_i)+\rho_s(y,p_j).
$$
We sum these inequalities over $s=1,\dots,n$. In the right-hand side, the sums
$\sum_s \rho_s(x,p_i)=n\rho(x,p_i)$,
$\sum_s \rho_s(y,p_j)=n\rho(y,p_j)$ appear, each not exceeding
$\eps n$. Integrating over $X_i\times X_j$
and summing over all $i,j=1,2,\dots,k$ yields
$$
\sum_s \iint_{X\times X} |\rho_s-d_s|\leq \delta(\eps) n+
2\eps n.
$$
Now assume that $\delta(\eps)+2\eps<c/10$.
Then the estimate
$\|\rho_s-d_s\|<c/5$ holds  at least for $n/2$ indices $s$.

Note that all metrics $d_s$ lie in the same space $L$
of piecewise constant functions, which has dimension
$k^2+1$. Besides, their norms are bounded by a constant depending only
on the uniform bound on the norms of semimetrics in $M$. It follows that
if $n$ is sufficiently large, then among any
$n/2$ of these metrics there are two, say
$d_s,d_t$, with distance at most
$c/5$ from each other (indeed, otherwise the balls in $L$ of radius
$c/10$ centered at these functions would be
disjoint and would lie in a ball
of a bounded radius, which is impossible for large $n$
from volume considerations; note that the bound on $n$ here depends
only on the dimension of the space, but not on its structure). But if
$\|\rho_s-d_s\|<c/5$,
$\|\rho_t-d_t\|<c/5$, $\|d_s-d_t\|<c/5$, then $\|\rho_s-\rho_t\|<c$,
contradicting the assumption.
\end{proof}

Note that the criterion may be rephrased for not
neccesarily
convex family of semimetrics: $\eps$-entropies of
all finite convex combinations must be uniformly bounded,
and if it is the case, then the family is precompact.
It immediately follows from Theorem \ref{kriko}
and the fact that the set in Banach space is precompact
if and only if its convex hull is precompact.

In the following special case we see that not even
all convex combinations are necessary for
assuring in precompactness.

\begin{theorem}\label{kriauto}
 Let $(X,\rho)$ be admissible semimetric triple, $T$ be
measure-preserving transform on $X$
(not necessarily invertible). 
 Denote $T^k\rho(x,y)=\rho(T^kx,T^ky)$
and $\rho_n^{av}=n^{-1}\sum_{k=1}^n T^k\rho$.
Assume that for any $\eps>0$ $\eps$-entropies of
semimetrics $\rho_n^{av}$ are uniformly bounded. Then the orbit
$\{\rho, T\rho, T^2\rho,\dots\}$
of $\rho$ under action of $T$ is precompact (say, in m-norm).
\end{theorem}

\begin{proof}
 Assume the contrary, then for some $\eps>0$ and
some positive integers $n_1<n_2<\dots$
the mutual distances between metrics $T^{n_i}\rho$
are not less than $\eps$. We know from the proof
of Theorem \ref{kriko} that
there exists dimension $D$ depending on $\eps$ and,
if $n$ is large enough, there exists a subspace
$L_D$ of dimension $D$ such that not less than, say,
$n/2$
metrics $T^i \rho$ ($i=1,2,\dots,n$) are $\eps/9$-close
to $L_D$. Also ball of radius, say,
$2\int \rho+2\eps$ in $L_D$ has $\eps/9$-net of cardinality at most
$C=C(\eps,\rho)$. Hence we may find at least $n/2C$
indices $i_1<i_2<\dots<i_k\leq n$, $k\geq n/2C$
such that mutual distances between metrics $T^{i_s}\rho$
do not exceed $\eps/3$.

Consider pairs of integers
$(a,p)$, where $1\leq a\leq k$, $1\leq p\leq M$, $M=5C+1$. Then all sums $i_a+n_p$ are less than
$2n$ (if $n$ is large enough), while there are more than $2n$
such sums. Then by pigeonhole principle
there exist $i_a<i_b$ and $n_p<n_q$
such that $i_a+n_p=i_b+n_q$.
Hence the distance
between metrics $T^{i_a}\rho$ and $T^{i_b}\rho$ coincides with the distance between $T^{n_p}\rho$ and $T^{n_q}\rho$, while the latter
is not less than $\eps$ and the former is not greater than
$\eps/3$. A contradiction.
\end{proof}

\subsection{Matrix definitions of admissible metrics}

Using Lemma~\ref{cut}, one can characterize the admissibility of
(summable) metrics in terms of the behavior of the traces of the
matrices of block averages of metrics.

\begin{theorem}\label{projective}
Let $\rho$ be a measurable summable metric
defined on a Lebesgue space
$(X,\mu)$. Consider a partition $\lambda$ of $X$ into $n$ sets of
equal measure,
$X=\sqcup_{i=1}^n \Delta_i$,
and construct the matrix
$A_{\rho,\lambda}$ of averages of
$\rho$ over $\lambda$:
$$
A_{\rho,\lambda}(i,j)=n^2 \int_{\Delta_i\times \Delta_j} \rho d\mu^2.
$$

{\rm1)} If
$$
\inf \frac1n \tr A_{\rho,\lambda}=0,
$$
where the infimum is taken over all $n$ and over all partitions of
$X$ into $n$ parts of equal measure, then the metric
$\rho$ is admissible.

{\rm2)} Assume that the metric $\rho$ is admissible and a sequence
of partitions $\lambda_1,\lambda_2,\dots$ satisfies the Lebesgue
density theorem (i.e., for every measurable subset
$Y\subset X$, for almost every point
$y\in Y$, the density of
$Y$ in the element $\lambda_k(y)$
of the partition $\lambda_k$ that contains $y$ tends to $1$ as $k\to +\infty$). Then
$$
\lim_{k\rightarrow +\infty}
\frac1{n_k} \tr A_{\rho,\lambda_k}=0,
$$
where $n_k$ is the number of parts in
$\lambda_k$. This property is satisfied, for example, for a sequence
of dyadic partitions, for partitions of an interval into equal
subintervals, partitions of a square into equal rectangles, etc.
\end{theorem}

\begin{proof} 1) If  $\rho$ is not admissible, then, by
Theorem~\ref{kritdop}, there exist $c>0$
and a measurable set $Y$
of measure $\mu(Y)\geq c$ such that $\rho(x,y)\geq c$
for almost all pairs $x,y\in Y$. Put
$m_k=\mu(\Delta_k\cap Y)$. Then
$$
n^2\int_{\Delta_k^2} \rho d\mu^2\geq
cn^2m_k^2;
$$
summing over $k$ yields
$$
\tr A_{\rho,\lambda}\geq cn^2\sum_{k=1}^n m_k^2\geq cn (\sum_{k=1}^n m_k)^2
\geq c^3 n,
$$
so that the infimum in question is not less than
$c^3$, a contradiction.

2) First consider arbitrary $\rho$ and $\lambda$.
Averaging the triangle inequality
$\rho(x,y)\leq \rho(x,z)+\rho(y,z)$
over $x,y\in \Delta_k$, $z\in \Delta_m$ yields
$A_{\rho,\lambda}(k,k)\leq 2A_{\rho,\lambda}(k,m)$.
Now, averaging over the pairs
$k,m$, we see that
$$
\frac 1n \tr A_{\rho,\lambda}\leq 2\|\rho\|_{L_1}.
$$
This immediately implies that
$$|\frac1n \tr A_{\rho,\lambda}-
\frac1n \tr A_{\rho',\lambda}|\leq 2\|\rho-\rho'\|_m.$$
Since every summable admissible semimetric can be approximated in the
m-norm by its cut-offs
(Lemma~\ref{cut}), it suffices to prove the required assertion under the
assumption that the semimetric
$\rho$ is bounded.

Fix $\eps>0$ and find a partition
$X=\sqcup_{i=0}^N X_i$ of $X$ into a set
$X_0$ of measure less than $\eps$ and sets $X_1,\dots,X_N$ of
$\rho$-diameter less than $\eps$. That of the sets $X_i$ which contains
$y\in X$ will be denoted by $X(y)$, by analogy with
$\lambda(y)$. The Lebesgue density theorem
(more exactly, its assumption) implies the following:
the measure of the set of points $y$ for which
$$
\mu(X(y)\cap \lambda_k(y))\leq \frac1{2n_k}
$$
tends to zero as $k$ tends to infinity. Take the union of the set of such exceptional $y$'s with
$X_0$ and call the obtained set
$Y_0$ (here $Y_0$ depends on $k$ and has measure
$<\eps$ for large $k$). Put $n=n_k$ and denote the elements of the partition $\lambda_k$
by $\Delta_1,\dots,\Delta_n$.
In $\cup_{j=1}^n \Delta_j^2\subset X^2$
consider the set $E$ of points $(x,y)$ such that $x\in Y_0$
or $y\in Y_0$. Obviously,
$\mu^2(E)\leq 2\frac 1n \mu(Y_0)$.
Put $E_1=\cup \Delta_j^2\setminus E$.
Note that on $E_1$ the semimetric $\rho$
does not exceed
$\eps$ pointwise. Indeed, let $x,y\in \Delta_j$,
$x,y\notin Y_0$, $x\in X_i$, $y\in X_l$.
Then necessarily $i=l$, since otherwise summing up the inequalities
$\mu(\Delta_j\cap X_s)>\frac 1{2n}$
for $s=i,l$ leads to a contradiction. Thus
$$\int_{E_1} \rho d\mu^2\leq \eps \mu^2(E_1)\leq \eps/n$$
and $$\int_{E}\rho d\mu^2 \leq\, \dm{\rho} X \mu^2(E)\leq
2\eps\, \dm{\rho}X/n.
$$
Adding these two inequalities and recalling that $\eps$ is arbitrary yields
$$
\int_{\cup\Delta_j^2} \rho d\mu^2 =o(1/n),
$$
as required.
\end{proof}

Let $x_1,\dots,x_n$ be points chosen at random and independently from $X$.
The classification theorem \cite{Riem,umn04} says that a metric
triple is determined up to isomorphism by the corresponding
distribution of the distance matrices
$\rho(x_i,x_j)_{1\leq i,j\leq n}$ (for all $n$). Therefore, the
admissibility of a metric must also be expressible in terms
of this distribution. Among various ways to give such a description, we confine
ourselves to the following one.

\begin{theorem}
{\rm1)} If a metric $\rho$ is not admissible, then there exists
$c>0$ such that the probability of the following event tends to one as
$n$ tends to infinity:

\medskip
{\rm($P_c$)} there is a set of indices
$I\subset \{1,2,\dots,n\}$ of cardinality at least $cn$ such that
$\rho(x_{i},x_{j})\geq c$ for all distinct
$i,j\in I$.
\medskip

{\rm2)} If a metric $\rho$ is admissible, then for every
$c>0$ the probability of
$P_c$ tends to zero.

In both cases, the rate of convergence to $1$ or $0$ is at least exponential
in $n$.
\end{theorem}

\begin{proof}
1) Find $c>0$ and a measurable set $Y\subset X$ of measure $2c$
such that $\rho(x,y)\geq c$ for almost all pairs $x,y\in Y$. Then on
the average $Y$ contains
$2cn$ points among $x_1,\dots,x_n$, and the probability that the number
of such points is at most
$cn$ tends to $0$ exponentially in $n$ (by standard large deviations estimates in the Law of Large Numbers for Bernoulli independent summands). The probability that a pair of such
points is at distance at most
$\eps$ is zero. Therefore, with probability tending to one
exponentially, a required set of indices does exist.

2) Let $\rho$ be an admissible metric. Partition $X$ into a set
$X_0$ of measure
$<c/2$ and sets $X_1,\dots,X_N$ of
$\rho$-diameter $\leq c/2$. Note that if a required set of indices $I$
is found, then for every  $i=1,\dots,N$ the point
$x_k$ lies in $X_i$ for at most one index
$k\in I$. Therefore, for
$n>10N/c$ this implies that at least
$2cn/3$ points among $x_1,\dots,x_n$ fall into
$X_0$. But, again, this happens with probability exponentially small
in $n$.
\end{proof}

\medskip
{\bf Remark.} In conclusion of this section, we mention
an important problem
from the theory of metric measure spaces.

We define an integral averaging operator as follows. Let
$\rho \in L^2_{\mu \times \mu}(X\times X)$. Consider the following linear operator
$I_{\rho}\equiv I$:
$$ 
I(f)(y) \equiv \int_X \rho(x,y)f(x)d\mu(x),$$
where $f \in L_{\mu}^2(X)$. Roughly speaking, this operator measures
the weighted average distance between the points of the space.

Obviously, $I$ is a self-adjoint Hilbert--Schmidt operator in
$L_{\mu}^2(X)$. It is of great interest to study its spectrum and, in
particular, the leading eigenvalues. It may happen that some metric
invariants of an action of a group $G$ on $X$ can be
expressed in terms of
joint characteristics of the operator $I$ and the unitary operators
$U_g$, $g\in G$. Since the spectrum of the random distance matrix is
a complete invariant of
an admissible triple, it is of interest to
study this spectrum and compare it with the spectrum of the averaging
operator~$I$.

\newpage

\section{The dynamics and $\eps$-entropy
of admissible metrics;
discreteness of the spectrum.}

\subsection{Scaling entropy and the statement of the discreteness
criterion}

The theory of admissible metrics and semimetrics which we considered
in the first chapter, being of interest in itself, also leads to new
applications to ergodic theory. These applications rely on replacing
the dynamics of measure-preserving transformations in the original
measure spaces by the dynamics of the associated transformations in
the spaces of admissible metrics. This should be compared with the
transition $T \mapsto  U_T$ from measure-preserving transformations
to unitary operators in $L^2$ in the early 1930s. Let $T$ be a
transformation of a Lebesgue space $(X,\mu)$ preserving the measure
$\mu$; then we can consider the transformation $R_T$ of the cone of
admissible metrics ${\mathcal A}{\rm dm}(X,\mu)$ defined by the
formula $R_T(\rho)(x,y)=\rho(Tx,Ty)$; The set of the admissible
metrics of type $R_T^n(\rho)(x,y)\equiv \rho_n(x,y)=\rho(T^nx,T^ny);
n\in \mathbb Z$ we called $T$-orbit of $\rho$. Introduce the
averaging operator $M_n=\frac{1}{n}\sum_{k=0}^{n-1} R_T^k$:
$$(M_n\rho)(x,y)=\frac{1}{n}\sum_{k=0}^{n-1}\rho(T^k x,T^k y).$$
 It is clear that $M_n$ sends every
semimetric $\rho$ to a new semimetric $\rho_n^{av}:=M_n\rho$, and we are
interested in the study of its properties as $n$ tends to infinity.

 In fact, we study the action of the
unitary operator $U_T\bigotimes U_T$ and averages of its powers.
However, the crucial point is that we consider this action on the
cone of admissible metrics rather than simply in $L^1$.

Recall the definition of the scaling entropy of an automorphism
introduced in \cite{usp,Mark} (see also \cite{VGor}).

\begin{definition}
Let $T$ be an automorphism of a Lebesgue space $(X,\mu)$.
For an arbitrary $\eps>0$ and an arbitrary semimetric $\rho$, we define the
\emph{class of scaling sequences} for the automorphism $T$ and the semimetric
$\rho$ as the family of all nondecreasing sequences
$\{c_n\}$ such that
 $$0<\liminf\limits_{n\to \infty}
\frac{\Ent{\eps}{\rho_n^{av}}}{c_n}\leq
\limsup\limits_{n\to \infty}\frac{\Ent{\eps}{\rho_n^{av}}}{c_n}<\infty.$$

All sequences in the same class are equivalent.
If the limit exists, it is called the \emph{scaling
$\eps$-entropy} of $T$ with respect to the semimetric $\rho$
and scaling sequence $\{c_n\}$.
Finally,
if the limit of these $\eps$-entropies as $\eps \rightarrow 0$
exists with some normalization in $\eps$, then
it is called the \emph{scaling entropy} of $T$
(with respect to the semimetric $\rho$, scaling sequence
and normalization).
 \end{definition}

In the calculations performed so far in concrete
examples, the latter limit does
exist and does not depend on the choice of an admissible metric. A
special role is played by the
class of bounded nondecreasing scaling sequences.

The main result of this paper is the following theorem.

\begin{theorem}
\label{pointspect}
Let $T$ be a measure-preserving
automorphism of a Lebesgue space
$(X,\mu)$. Then the following conditions are equivalent:
\begin{itemize}
\item[\rm1)] $T$ has a purely discrete spectrum.
\item[\rm2)] For every admissible semimetric
$\rho \in L^1(X^2)$ and every $\eps>0$, the scaling sequences are
bounded.
\item[\rm3)] For some admissible metric
$\rho \in L^1(X^2)$ and every $\eps>0$, the scaling sequences are
bounded.
\end{itemize}
\end{theorem}

\begin{remark}\label{avmetric}
 By
individual ergodic theorem the limiting average semimetric
$$\rho^{av}=\lim_{n \to \infty}\rho_n^{av}(x,y)$$
does exist almost everywhere. Results of Chapter
1 show that it is admissible if and only
if for any $\eps>0$ the scaling sequences of $\rho_n^{av}$
are uniformly bounded by $n$
(``if'' part follows from Lemma \ref{lem3},
``only if'' part from Corollary \ref{cor2}).
It allows to reformulate Theorem \ref{pointspect},
replacing conditions 2) to \textit{2')
For every admissible semimetric
$\rho \in L^1(X^2)$ the imiting average metric $\rho^{av}$ is
admissible}; analagously for condition 3).
\end{remark}

The implication $2)\Rightarrow 3)$ is trivial,
and the proof of the other two
ones is given below; the proof relies on the obtained results on
admissible metrics.

\subsection{Proof of the main theorem; the
implication $1)\Rightarrow 2)$.}

Here we use the result obtained in the first
chapter on the precompactness of a family of
admissible metrics in the m-norm.

Since automorphism $T$ has purely discrete spectrum, tensor square
of it - $T^{\otimes 2}$ (acting on $X\times X$) also has purely
discrete spectrum. It implies that the $T^{\otimes 2}$-orbit of any
function $f\in L^2(X\times X)$ is precompact. Take any admissible
semimetric $\rho$ on $X$.

Our nearest goal is to prove that $T$-orbit of
$\rho$ is precompact in $L^1(X\times X)$

Assume the contrary, then for some
$c>0$ and some infinite
subset ${\cal N}\subset \mathbb{N}$ we have
$\|\rho_n-\rho_k\|\geq c$ for all distinct $n,k\in {\cal N}$.
Choose large $M>0$ so that $\|\rho-\rho^{M}\|<c/3$,
where $\rho^M$ is a cut-off of $\rho$ on level $M$.
Since taking cut-off commutes with action of $T$,
we get
$$
\|\rho_n^M-\rho_k^M\|\geq \|\rho_n-\rho_k\|-\|\rho_n-\rho^M_n\|-
\|\rho_k-\rho_k^M\|\geq c-c/3-c/3=c/3
$$
for all $n,k\in {\cal N}$. Hence for a bounded metric
$\rho^M$ its $T$-orbit also has a separated infinite subset.
But it belongs to $L^2(X\times X)$, hence its orbit is precompact
even in $L^2$, and so in $L^1$. A contradiction.

So we see that $T$-orbit of $\rho$ is precompact in $L^1$, hence its
closure in $L^1$ is compact. But $\rho$ is admissible, hence by
Lemma \ref{lem3} the closure of $T$-orbit of $\rho$ contains only
admissible metrics. Then it is compact also in m-norm by Corollary
\ref{corcomp}. So, its convex hull is
precompact in m-norm. Hence $\eps$-entropies of the metrics from
this convex hull are uniformly bounded by Corollary \ref{cor2},
as desired.

\begin{remark}
Actually, the following more general fact is proved.
Discreteness
of spectrum of $T$ implies that $\eps$-entropies of all convex
combinations of semimetrics in $T$-orbit of a given admissible
semimetric $\rho$ are uniformly bounded (but not only
for averages over initial segments).
\end{remark}

A typical and by von Neumann classical theorem general
example of the transformation with discrete spectrum
is a rotation on a compact abelian group. By
Remark \ref{avmetric}
and already proved part of Theorem \ref{pointspect}
we see that averaged (over orbit of the rotation)
metric is then admissible. It is clear that instead
of averaging over orbit of
a rotation we can consider the averaging over the closure
of the orbit, which coincides with
the whole group in the ergodic case.
Below we prove the analog of that fact for general
(not necessary Abelian) compact group. The proof
is very similar to the above proof of part
of Theorem \ref{pointspect}. Also, we
prove that admissible rotation-invariant metric
must be continuous.

\begin{statement}
For an arbitrary admissible metric $\rho$ on a compact group $G$
endowed with Haar measure, the average of the metric $\rho$ with
respect to the compact subgroup of
the group of translations is admissible. The
average over whole group is, moreover, invariant,
and hence continuous.
\end{statement}

\begin{proof}
Note that the map $G\rightarrow L^1(G^2)$:
$g\rightarrow \rho_g(x,y):=\rho(gx,gy)$ is continuous
(by continuity of rotation in mean). Hence its
image $I$ is compact in $L^1$. Then it is compact
also in m-norm
by Corollary \ref{corcomp}.
Then its convex hull is precompact in m-norm and
so $\eps$-entropies of its elements are uniformly bounded
by Corollary \ref{cor2}.
The averaged metric $\int_H \rho_g d\mu_H$
(where $H$ is a compact subgroup of $G$, $\mu_H$ is
Haar measure on $H$)
lies in the closed (say, in $L^1$)
convex hull of $I$ and hence is admissible by Lemma
~\ref{lem3}. Now we
will show that the averaged metric over whole $G$
is continuous. Since this metric is
translation-invariant, it suffices to prove that it is
continuous at unity. The admissibility
criterion (Theorem~\ref{kritdop}) says
that for almost all
$x \in G$, the ball $B=\{y\in G \colon \rho(x,y)\leq r\}$ of radius $r>0$
centered at $x$ has positive measure.
But then by Steinhaus theorem (see, for example \cite{Str})
the set
$B\cdot B^{-1}$ contains a neighborhood of unity,
and for every
$z\in B\cdot B^{-1}$, by the triangle inequality and
the invariance of $\rho$, we have
$\rho(1,z)\leq 2r$, which proves that the metric is continuous at
unity.
\end{proof}

\subsection{Proof of the implication $3)\Rightarrow 1)$}

Now we will prove 
implication
$3)\Rightarrow 1)$:
if there exists an admissible metric
$\rho$ such that the corresponding class of scaling sequences
consists of bounded sequences for every
$\eps>0$, then the automorphism
$T$ has a purely discrete spectrum. Clearly, one may
assume that $\rho$ is bounded by replacing it
to the cut-off if necessary.

We use the following known criterion of discreteness of spectrum for a unitary
operator $U$ in Hilbert space: \emph{$U$-orbit of any element is
precompact.} This is the corollary of the spectral theorem for
unitary operator. Recall that slightly more general fact is true:
$U$-orbit of $x$ is precompact if and only if $x$ lies in the closed
span of eigenvectors of $U$. Finally, for the
unitary operator corresponding to the
automorphism $T$ on the Lebesgue space $(X,\mathfrak{A},\mu)$,
this closed span is a space of functions
in $L^2$, measurable w.r.t. some $\sigma$-subalgebra
$\mathfrak{B}\subset \mathfrak{A}$
(or, in other words, the space of functions, constant
on almost all parts of some measurable partition $\xi$).
For the square-summable function of two
variables $f(x,y)$ on $X\times X$ precompactness of its
$T$-orbit $\{f(T^nx,T^ny),n=1,2,\dots\}$
therefore implies that $f$ is measurable with
respect to sub-algebra $\mathfrak{B}^2$. In particular,
for almost all $x$ functions $f(x,\cdot)$ are
$\mathfrak{B}$-measurable and for almost all parts
of corresponding partition $\xi$ the functions $f(x,\cdot)$
coincide a.e. for a.e. $x$ from this part. Assume
that it holds for the bounded
(or just square summable) admissible metric $f=\rho$.
But then for any two
points $u,v$ the functions $f(u,\cdot)$ and $f(v,\cdot)$
are different on the ball $B(u,\rho(u,v)/3)$, which has positive
measure for almost all $u$ by Theorem \ref{kritdop}.
In other words, for almost all $u$ there is no $v$
such that functions $f(u,\cdot)$ and $f(v,\cdot)$ coincide
a.e. (Such functions are called in \cite{Clas} ``pure functions of
two variables'', this property is important in the
classification theorem.)
It implies that partition $\xi$ is trivial and so the spectrum
of $T$ is purely discrete.

Now for finishing the proof of implication $1)\Rightarrow 3)$ it
suffices to combine above general techniques and Theorem
\ref{kriauto}.


\subsection{Further remarks}

\subsubsection{Relation to A-entropy}
In \cite{Kush}, another discreteness criterion for the spectrum of an
automorphism was proved; it is also based on the notion of
entropy (in that case, sequential, or A-, or
Kirillov--Kushnirenko entropy). According to this criterion, the
spectrum of an automorphism $T$ is discrete if and only if
\begin{equation}\label{ku}
\limsup_{n \to \infty}\frac{1}{n}H\big(\prod_{k=1}^n T^{i_k}\xi\big) =0
\end{equation}
for every finite partition $\xi$ and an arbitrary sequence
${i_1<i_2< \dots}$ of positive integers. Here
$H(\cdot)$ is the entropy of a finite partition.
One can easily check that the entropy
$H(\cdot)$ in criterion~(\ref{ku}) can be replaced with the
$\eps$-entropy (when $\eps>0$ takes all positive values).
Kushnirenko's proof is based on the following two reductions.

(1) The spectrum of $T$ is discrete if and only if the set of partitions
$\{T^n\xi\,|\,n=1,2,\dots\}$ is precompact
with respect to some natural metric on partitions. Since the number of
parts in the partition $T^n\xi$ is fixed, various natural metrics turn
out to be equivalent. For our purposes, it is convenient to consider
the distance in
$L^1$ or in the m-norm between the block semimetrics corresponding
to partitions.

(2) Such a family is precompact if and only if the normalized entropies
(\ref{ku}) tend to zero.

The product of partitions appearing in~(\ref{ku}) corresponds to
the maximum of the associated block metrics. However,
our main Theorem~\ref{pointspect} involves averages of
semimetrics. So, the precompactness criterion (2)
is to be compared with our Theorem \ref{kriko}
It is not
a complete analog of Kushnirenko's criterion:
first, it applies to general admissible semimetrics; second,
deals with averages rather than maxima; third,
uses the entropy of metrics
rather than partitions. In the particular case
where one deals with the
convex hull of a family of cut semimetrics corresponding to
partitions into two parts of equal measure, Kushnirenko's
criterion follows from
the condition of criterion in Theorem \ref{kriko}.
In the general situation, the relation
between two criteria is not quite clear; for instance, we do not know
any exact generalization of Kushnirenko's criterion
to the case of general semimetrics.

\subsubsection{Conjectures}

The asymptotics of
the scaling entropy for an arbitrary automorphism is not known.
Most probably, in the
other extreme case, i.e., for actions with positive Kolmogorov
entropy, the answer can be obtained
in the same way as in the discrete spectrum case.
Namely, we state the following conjecture.

\begin{conjecture}
For any automorphism $T$ with positive entropy,
the scaling sequence has order $n$. In other words, for every
admissible metric $\rho$,
$$\lim_{\epsilon \to 0}\lim_{n \to \infty}\frac{\Ent{\eps}{\rho^T_n}}{\phi(\epsilon)n}=h(T),$$
where $\phi(\epsilon)$ is a function, possibly depending on
$\rho$, and $h(T)$ is the classical entropy of $T$.
\end{conjecture}

In \cite{Mark} we formulate a weaker conjecture
that the equality is true for generic admissible metric.
But it seems that using Shannon-McMillan-Breiman
theorem it is possible to prove above conjecture.

As to zero entropy --- it is not yet known what intermediate --- between bounded
and linear --- growth the scaling sequences for
automorphisms can have.
Most probably, logarithmic growth with different bases can be
achieved (for
oricycles, adic transformations, etc.). For arbitrary
groups, the growth of scaling sequences lies between bounded growth and the growth
of the number of words of given length in the group. For the groups
$\sum_1^{\infty} {\Bbb Z}_p$, examples are already
found  in \cite{VGor, usp} where the scaling entropy
grows as an arbitrary integer power of the logarithm of the number of words
of given length. It is still plausible that the growth does not depend on the choice of
admissible metric.

However, recall that entropy characteristics are just the simplest
(``unary,'' or ``dimensional'') invariants of the dynamics of metrics.
There are other
asymptotic invariants of the sequence of average
metrics with respect to automorphism.

\end{document}